\documentclass[reqno, a4paper]{amsart}

\usepackage[T1]{fontenc}
\usepackage[utf8]{inputenc}
\usepackage[british]{babel}
\usepackage{graphicx}
\usepackage{amsfonts}
\usepackage{amssymb}
\usepackage{amsthm}
\usepackage{amsmath}
\usepackage[all]{xy}
\usepackage{hyperref}
\usepackage{mathbbol}
\usepackage{mathabx}
\usepackage{color}
\usepackage{todonotes}
\usepackage[nointegrals]{wasysym}

\theoremstyle{plain}

\newtheorem{theorem}[subsection]{Theorem}
\newtheorem{lemma}[subsection]{Lemma}
\newtheorem{proposition}[subsection]{Proposition}
\newtheorem{corollary}[subsection]{Corollary}

\theoremstyle{definition}

\newenvironment{tfae}
{
\begin{enumerate}}
{\end{enumerate}}

\newcommand{\comp}{\raisebox{0.2mm}{\ensuremath{\scriptstyle{\circ}}}}
\newcommand{\defn}{\textbf}

\newcommand{\noproof}{\hfill \qed}
\newcommand{\To}{\Rightarrow} 

\newcommand{\mus}[2]{\lgroup #1 \; #2 \rgroup}
\newcommand{\musv}[2]{\bigl\lgroup\begin{smallmatrix} #1 \\ #2 \end{smallmatrix}\bigr\rgroup}
\newcommand{\musvv}[2]{\Biggl\lgroup\begin{smallmatrix} #1 \\ \vdots\\ \\ #2 \end{smallmatrix}\Biggr\rgroup}
\newcommand{\defeq}{\coloneq}

\newcommand{\cosmash}{\diamond} 
\newcommand{\normal}{\lhd} 
\newcommand{\join}{\vee} 
 
\newcommand{\gnab}{\text{\rm \textexclamdown}}

\DeclareMathOperator{\bigdiamond}{\raisebox{-0.9mm}{\scalebox{1.5}{$\Diamond$}}}
\DeclareMathOperator{\cod}{Cod}

\DeclareMathOperator{\dom}{Dom}

\DeclareMathOperator{\Ker}{Ker}
\DeclareMathOperator{\D}{D}

\renewcommand{\L}{\mathrm{L}}

\DeclareMathOperator{\op}{op}

\DeclareMathOperator{\Cr}{Cr}

\newcommand{\CC}{\ensuremath{\mathbb{C}}}
\newcommand{\DD}{\ensuremath{\mathbb{D}}}

\newcommand{\KK}{\ensuremath{\mathbb{K}}}

\newcommand{\PP}{\ensuremath{\mathcal{P}}}
\newcommand{\VV}{\ensuremath{\mathcal{V}}}

\newcommand{\Alg}{\ensuremath{\mathsf{Alg}}}

\newcommand{\Fun}{\ensuremath{\mathsf{Fun}}}

\newcommand{\Arr}{\ensuremath{\mathsf{Arr}}}
\newcommand{\Ext}{\ensuremath{\mathsf{Ext}}}

\newcommand{\Pt}{\ensuremath{\mathsf{Pt}}}
\newcommand{\Nil}{\ensuremath{\mathsf{Nil}}}

\newcommand{\NH}{\ensuremath{{\rm (NH)}}}
\newcommand{\SH}{\ensuremath{{\rm (SH)}}}

\author{Cyrille Sandry Simeu}
\author{Tim Van~der Linden}

\address[Cyrille Sandry Simeu, Tim Van~der Linden]{Institut de
Recherche en Math\'ematique et Physique, Universit\'e catholique
de Louvain, che\-min du cyclotron~2 bte~L7.01.02, B--1348
Louvain-la-Neuve, Belgium}

\email[Cyrille Sandry Simeu]{cyrille.simeu@uclouvain.be}
\email[Tim Van~der Linden]{tim.vanderlinden@uclouvain.be}

\thanks{The second author is a Research Associate of the Fonds de la Recherche Scientifique--FNRS}

\subjclass[2010]{17B30, 18D35, 18E10, 18G50, 20F18, 20J06}
\keywords{Semi-abelian, algebraically coherent category; nilpotency; Higgins commutator; cross-effect.}

\def\pullback{% with thanks to Valerian Even
 \ar@{-}[]+R+<6pt,-1pt>;[]+RD+<6pt,-6pt>%
 \ar@{-}[]+D+<1pt,-6pt>;[]+RD+<6pt,-6pt>}

\def\dottedpullback{%
 \ar@{.}[]+R+<6pt,-1pt>;[]+RD+<6pt,-6pt>%
 \ar@{.}[]+D+<1pt,-6pt>;[]+RD+<6pt,-6pt>}

\newdir{>}{{}*:(1,-.2)@^{>}*:(1,+.2)@_{>}}
\newdir{<}{{}*:(1,+.2)@^{<}*:(1,-.2)@_{<}}

\newdir{>>}{{}*!/3.5pt/:(1,-.2)@^{>}*!/3.5pt/:(1,+.2)@_{>}*!/7pt/:(1,-.2)@^{>}*!/7pt/:(1,+.2)@_{>}}
\newdir{ >>}{{}*!/8pt/@{|}*!/3.5pt/:(1,-.2)@^{>}*!/3.5pt/:(1,+.2)@_{>}}
\newdir{ |>}{{}*!/-3.5pt/@{|}*!/-8pt/:(1,-.2)@^{>}*!/-8pt/:(1,+.2)@_{>}}
\newdir{ >}{{}*!/-8pt/@{>}}

\begin{document}

\title[On the ``three subobjects lemma'']{On the ``three subobjects lemma''\\ and its higher-order generalisations}

\begin{abstract}
We solve a problem mentioned in the article~\cite{BeBou} of Berger and Bourn: we prove that in the context of an algebraically coherent semi-abelian category, two natural definitions of the lower central series coincide. 

In a first, ``standard'' approach, nilpotency is defined as in group theory via nested binary commutators of the form $[[X,X],X]$. In a second approach, higher Higgins commutators of the form $[X,X,X]$ are used to define nilpotent objects~\cite{HVdL, Actions, Higgins}. The two are known to be different in general; for instance, in the context of loops, the definition of Bruck~\cite{Bruck} is of the former kind, while the \emph{commutator-associator filtration} of Mostovoy~\cite{Mostovoy1, Mostovoy2} and his co-authors is of the latter type. Another example, in the context of Moufang loops, is given in~\cite{BeBou}.

In this article, we show that the two streams of development agree in any algebraically coherent semi-abelian category. Such are, for instance, all \emph{Orzech categories of interest}~\cite{Orzech}. Our proof of this result is based on a higher-order version of the \emph{Three Subobjects Lemma} of~\cite{acc}, which extends the classical \emph{Three Subgroups Lemma} from group theory to categorical algebra. It says that any $n$-fold Higgins commutator $[K_1, \dots,K_n]$ of normal subobjects $K_i\normal X$ may be decomposed into a join of nested binary commutators. 
\end{abstract}

\maketitle

\section{Introduction}
In his 1956 article~\cite{Higgins}, Higgins introduced a commutator $[K_1, \dots,K_n]\leq X$ of~$n$ subobjects $K_1$, \dots, $K_n$ of an object $X$, for any $n\geq 2$. Remarkable about this definition of his is, that it is not biased towards the case $n=2$; in particular, it is far from true that $[K_1, \dots,K_n]$ can always be obtained via nested binary commutators of the form $[[\cdots [[K_1,K_2],K_3], \dots,K_{n-1}],K_n]$. On the other hand, Higgins shows that ternary commutators are in some sense unavoidable, since they occur naturally amongst binary ones, for instance in the join decomposition formula
\[
[K_1,K_2\join K_3]=[K_1,K_2]\join [K_1,K_3]\join [K_1,K_2,K_3]\leq X.
\]

His study of the commutator takes place in the context of \emph{varieties of $\Omega$-groups}, which are pointed varieties of universal algebras, whose theory contains the operations and identities of the theory of groups. As it turns out, the concept is not limited to this setting. The main definitions (which we shall recall below in~\ref{Definition Commutator}) based on the concept of a \emph{co-smash product}~\cite{Smash} or a \emph{cross-effect}~\cite{Hartl-Vespa} can be made in pointed regular categories with finite coproducts~\cite{MM-NC, Actions, HVdL}. Results such as the join decomposition formula hold in any finitely cocomplete \emph{homological} category (in~the sense of Borceux--Bourn~\cite{Borceux-Bourn}). In particular, they are valid in all Janelidze--Márki--Tholen \emph{semi-abelian} categories~\cite{Janelidze-Marki-Tholen}, which by definition are pointed with binary sums, Barr-exact (=~regular, and such that every equivalence relation is a kernel pair), and Bourn-protomodular (=~the Split Short Five Lemma holds).

Higgins commutators can be used, for instance, to express centrality of higher extensions~\cite{ESVdL}; these occur in the context of semi-abelian categories, in the Hopf formulae for homology via Galois theory~\cite{EGVdL} and in an interpretation of cohomology with trivial coefficients~\cite{RVdL2}. They also appear in the description of internal crossed modules~\cite{Janelidze} given in~\cite{HVdL}, and are closely related to the treatment of internal actions via cosmash products established in~\cite{Actions}.

The difference between biased and unbiased $n$-ary commutators is reflected in two distinct approaches towards the concept of \emph{nilpotency}---see~\cite{BeBou} where this is explored in detail. The \emph{lower central series} is either chosen to consist of nested binary commutators of the form $[[X,X],X]$---this is the approach followed, for example, by Higgins in the context of $\Omega$-groups~\cite{Higgins}, and by Bruck in the context of loops~\cite{Bruck}---or, alternatively, its terms are higher Higgins commutators of the form $[X,X,X]$, as in~\cite{HVdL, Actions, Higgins}. For instance, the \emph{commutator-associator filtration} of Mostovoy~\cite{Mostovoy1, Mostovoy2} and his co-authors is of the latter type---and not of the former, as explained for instance in~\cite{M-PI-SH}. An example, in the context of Moufang loops, showing that the two approaches need not agree is given in~\cite{BeBou}. This naturally leads to the question, under which conditions on the surrounding category they \emph{do} agree.

Our aim in this paper is to provide an answer to that question: to prove that \emph{algebraic coherence}~\cite{acc} is a sufficient condition for this to happen. This is a fairly well-studied property satisfied by many semi-abelian categories, including all \emph{Orzech categories of interest}~\cite{Orzech}, excluding loops and non-associative rings. In particular, the categories of groups, (commutative) rings (not necessarily unitary), Lie algebras over a commutative ring with unit, Poisson algebras and associative algebras are all examples, as are all varieties of such algebras, and crossed modules over those. Before going into further details, let us briefly describe our general strategy towards this result. 

In Categorical Algebra, several approaches to commutator theory exist. One reason Higgins commutators are relevant is, because they make us better understand the relationship between these different approaches. For instance~\cite{HVdL}, the compatibility between Smith commutators and Huq commutators (the so-called \emph{Smith is Huq} condition \SH\ of~\cite{MFVdL}, see also~\cite{Borceux-Bourn}) can be expressed as the condition that the commutator inequality
\[
[K,L,X]\leq [[K,L],X]\join [K,L]\leq X
\]
holds for all $K$, $L\normal X$.

In the case of groups, the validity of this commutator inequality may be viewed as a consequence of the classical \emph{Three Subgroups Lemma}, which says that 
\[
[[L,M],K]\leq [[K,L],M]\join [[M,K],L]\leq X
\]
whenever $K$, $L$ and $M$ are normal subgroups of a group $X$. Since it can be shown that in the category of groups, the ternary commutator $[K,L,M]$ of $K$, $L$, $M\normal X$ decomposes as 
 \[
 [[K,L],M]\join[[L,M],K]\join [[M,K],L],
 \]
the \emph{Three Subgroups Lemma} implies
\begin{align*}
[K,L,X]&= [[K,L],X]\join[[L,X],K]\join [[X,K],L]\\
&\leq [[K,L],X]\join [[X,K],L]\\
&\leq [[K,L],X]\join [K,L],
\end{align*}
because, by Proposition~\ref{Normality via commutator} below, $[X,K]\leq K$ since $K$ is normal in $X$.

With Theorem~7.1 in the article~\cite{acc}, this is made categorical as follows:
\begin{theorem}[The Three Subobjects Lemma]\label{Three Subobjects Lemma}
If $K$, $L$ and $M$ are normal subobjects of an object $X$ in an algebraically coherent semi-abelian category, then
\[
[K,L,M]= [[K,L],M]\join [[M,K],L]
\]
as subobjects of $X$. In particular, $[[L,M],K]\leq [[K,L],M]\join [[M,K],L]$.\noproof
\end{theorem}
As an immediate consequence, we see that in an \emph{algebraically coherent semi-abelian} category, the equality $[X,X,X]=[[X,X],X]$ holds for any object $X$. This naturally leads to the main questions of our article: \emph{What about $n$-fold commutators for $n\geq 3$?} In particular, 
\begin{quote}
\emph{Is there an ``$n$ Subobjects Lemma'' of which the Three Subobjects Lemma is a special case?} 	
\end{quote}
Our aim is to give an affirmative answer to this question, under the same conditions. (In the article~\cite{BeBou}, Berger and Bourn show that $[X,X,X]=[[X,X],X]$ in the weaker setting of \emph{algebraically distributive} categories; we do not know how to extend their proof to higher-order commutators.) We obtain Theorem~\ref{Theorem n Subobjects Lemma}, which says that given $n\geq 3$ and any choice of normal subobjects $K_1$, \dots, $K_n$ of an object $X$ in an algebraically coherent semi-abelian category, the commutator $[K_1, \dots,K_n]$ decomposes as a join 
\begin{multline*}
[K_1, \dots,K_n]=[[K_1,K_2],K_3, \dots,K_n]\vee [K_2,[K_1,K_3],K_4, \dots,K_n]\vee\cdots\\
\cdots\vee [K_2,K_3, \dots,K_{n-1},[K_1,K_n]].	
\end{multline*}

It is clear that, via an induction argument, this solves the ``nilpotency problem''. Indeed, the first statement of Theorem~6.24 in~\cite{BeBou} states the following, of which we are not going to analyse all the details here. Our Corollary~\ref{Corollary Nilpotency} says that condition (iv) holds in any algebraically coherent semi-abelian category. So, in this context also (i)--(iii) hold. 
\begin{theorem}
Let $\CC$ be a semi-abelian category. The following conditions are equivalent:
\begin{tfae}
\item the nilpotency tower of $\CC$ is homogeneous;
\item for each $n$, the $n$\textsuperscript{th} Birkhoff reflection $I_n\colon \CC \to \Nil_n(\CC)$ is of degree $\leq n$;
\item for each $n$, an object of $\CC$ is $n$-nilpotent if and only if it is $n$-folded;
\item for each object $X$ of $\CC$, iterated Huq commutator $[\cdots[[X,X],X], \dots, X]$ and
Higgins commutator $[X,X, \dots,X]$ of same length coincide.	\noproof
\end{tfae}
\end{theorem}
Theorem~\ref{Theorem n Subobjects Lemma} will help answering other open questions in Categorical Algebra as well. It leads, for instance, to Theorem~\ref{Theorem Decomposition into binary commutators}, which says that given $n\geq 3$ and any choice of normal subobjects $K_1$, \dots, $K_n$ of an object $X$ in an algebraically coherent semi-abelian category, the commutator $[K_1, \dots,K_n]$ decomposes as a join of binary commutators
\begin{equation*}
\bigvee_{\sigma\in S_n}[\cdots[[K_{\sigma(1)},K_{\sigma(2)}],K_{\sigma(3)}], \dots ,K_{\sigma(n)}]\leq X.
\end{equation*}
Current work-in-progress such as an interpretation of cohomology with non-trivial coefficients generalising~\cite{RVdL2} depends on this, as does a categorical version of the results of~\cite{MutluPorter1, MutluPorter2}. Let us now focus on the missing details in the statement of our theorem.

\subsection{Definition of the commutator}\label{Definition Commutator}
We start with the binary case, which was first treated in~\cite{MM-NC}. For the sake of simplicity, we work in a semi-abelian category $\CC$. Consider a cospan $(k\colon K\to X, l\colon L\to X)$. The \defn{Higgins commutator} $[K,L]\leq X$ is computed as in the commutative diagram
\[
\xymatrix{
 0 \ar[r] & K \cosmash L \ar@{{ |>}->}[r]^{\iota_{K,L}} \ar@{-{ >>}}[d] & K + L \ar@{-{ >>}}[r]^{r_{K,L}} \ar[d]^-{\mus{k}{l}} 
 & K \times L \ar[r] & 0 \\
 & [K,L] \ar@{ >->}[r] & X
}
\]
where $r_{K,L}=\bigl\lgroup\begin{smallmatrix}
	1_K & 0 \\ 0 & 1_L
\end{smallmatrix}\bigr \rgroup$ is the canonical morphism from the coproduct to the product, $\iota_{K,L}$ is its kernel and $[K,L]$ is the image of the composite $\mus{k}{l}\comp \iota_{K,L}$. The object $K\cosmash L$ is called the \defn{co-smash product}~\cite{Smash} of $K$ and $L$. Of key importance for us is the following result due to Mantovani and Metere, Theorem~6.3 in~\cite{MM-NC}:
\begin{proposition}\label{Normality via commutator}
In a semi-abelian category, a subobject $K\leq X$ is normal (we write $K\normal L$) if and only if $[K,X]\leq K$. \noproof
\end{proposition}
It is crucial here that the category is Barr-exact; this actually one of the reasons our results are formulated in the context of a semi-abelian category, rather than a merely homological one. Proposition~4.14 in~\cite{Actions} adds to Proposition~\ref{Normality via commutator} that the normal closure of $K\leq X$ may be obtained as the join $K\join [K,X]\normal X$. The \defn{Huq commutator}~\cite{Borceux-Bourn} is the normal closure of $[K,L]$ in~$X$; by the above, it is the join $[K,L]\join [[K,L],X]$. 

If the category $\CC$ is such that $[K,L]\normal X$ whenever $K$, $L\normal X$, then we say that~$\CC$ satisfies \defn{normality of Higgins commutators} (or condition \defn{(NH)} for short), see~\cite{CGrayVdL1}. All \emph{Orzech categories of interest}~\cite{Orzech} satisfy \NH. On the other hand, the category of (commutative) loops does not: as explained to us by Alan Cigoli, it is not hard to construct an explicit counterexample. In the case of algebras, the condition can be characterised more precisely as follows~\cite[Theorem 2.12]{GM-VdL2}, via a kind of weak associativity condition. Here $\KK$ is a field, and~$\Alg_{\KK}$ is the variety of \defn{non-associative algebras} over $\KK$, where an object is a $\KK$-vector space $V$ equipped with a bilinear operation $\cdot\colon V\times V\to V$ (which is not necessarily associative). In this context, a normal subobject $I\normal X$ is an ideal, and the commutator $[I,I]$ is $I^2$. 

\begin{theorem}\label{Theorem Algebras}
Let $\KK$ be an infinite field, and $\VV$ a subvariety of~$\Alg_\KK$. The following conditions are equivalent:
\begin{tfae}
	\item $\VV$ satisfies \NH;
\item there exist $\lambda_{1}$, \dots, $\lambda_{16}$ in $\KK$ such that
\begin{align*}
z(xy)=
\lambda_{1}y(zx)&+\lambda_{2}x(yz)+
\lambda_{3}y(xz)+\lambda_{4}x(zy)\\
&+\lambda_{5}(zx)y+\lambda_{6}(yz)x+
\lambda_{7}(xz)y+\lambda_{8}(zy)x
\end{align*}
and 
\begin{align*}
	(xy)z=
	\lambda_{9}y(zx)&+\lambda_{10}x(yz)+
	\lambda_{11}y(xz)+\lambda_{12}x(zy)\\
	&+\lambda_{13}(zx)y+\lambda_{14}(yz)x+
	\lambda_{15}(xz)y+\lambda_{16}(zy)x
\end{align*}
are identities in $\VV$;
\item $\VV$ is an \emph{Orzech category of interest}~\cite{Orzech}.\noproof
\end{tfae}
\end{theorem}
Note that associativity, or the Jacobi identity, are conditions as in (ii).

\subsection{Higher-order commutators}\label{Higher order}
We recall the definitions of~\cite{Smash, HVdL, Actions}. Take $n\geq 2$ and consider a collection of arrows $(k_i\colon {K_i\to X})_{1\leq i\leq n}$. The \defn{co-smash product} $K_1\cosmash\cdots\cosmash K_n$ of $K_1$, \dots, $K_n$ is the kernel of the arrow
\[
r_{K_1, \dots,K_n}\coloneq \musvv{\widehat\pi_1}{\widehat\pi_n}\colon K_1+\cdots+K_n\to \prod_{k=1}^n(K_1+\cdots+\widehat K_k+\cdots+K_n),
\]
where
\[
\widehat\pi_k\colon \coprod_{i=1}^nK_i\to \coprod_{\substack{i=1\\i\neq k}}^nK_i
\]
is such that $\widehat\pi_k\comp \iota_{K_i}$ is $\iota_{K_i}$ whenever $k\neq i$, and zero otherwise. The \defn{Higgins commutator} $[K_1,\dots,K_n]\leq X$ is computed as in the commutative diagram
\[
\xymatrix@C=3em{
 0 \ar[r] & K_1\cosmash\cdots\cosmash K_n \ar@{{ |>}->}[r]^-{\iota_{K_1,\dots,K_n}} \ar@{-{ >>}}[d] & \coprod_{i=1}^nK_i \ar[r]^-{r_{K_1,\dots,K_n}} \ar[d]^-{\lgroup k_1 \; \cdots \;k_n \rgroup} 
 & \prod_{k=1}^n\coprod_{\substack{i=1\\i\neq k}}^nK_i \\
 & [K_1,\dots,K_n] \ar@{ >->}[r] & X;
}
\]
it is the image of the composite $\lgroup k_1 \; \cdots \;k_n \rgroup\comp \iota_{K_1,\dots,K_n}$. 

In the category of groups, for $K$, $L$, $M\leq X$, a typical element of $[K,L,M]$ is of the form
\[
klk^{-1}l^{-1}mlkl^{-1}k^{-1}m^{-1},
\]
where $k\in K$, $l\in L$ and $m\in M$. In the category of loops, $k(lm)/(kl)m$ belongs to the commutator $[K,L,M]$, and in the variety $\Alg_\KK$, so does $(kl)m$.

Higgins commutators have excellent stability properties. Here we give a summary of those we shall need later on.

\begin{proposition}\cite{HVdL, Actions}\label{Higgins properties}
Suppose $K_i$, $L_i\leq X$ for $1\leq i\leq n$. Then we have the following (in)equalities of subobjects:
\begin{enumerate}
\item[(0)] if $K_i=0$ then $[K_1,\dots, K_i,\dots , K_n]=0$;
\item for $\sigma\in S_n$, $[K_1,\dots,K_n ]=[K_{\sigma(1)},\dots,K_{\sigma(n)}]$;
\item $f[K_1,\dots,K_n ]=[f(K_1),\dots,f(K_n)]\leq Y$ for $f\colon {X\to Y}$ any regular epi;
\item $[K_1,\dots,K_{i-1},L_i,K_{i+1},\dots ,K_n ]\leq [K_1,\dots,K_{i-1},K_i,K_{i+1},\dots,K_n]$ when $L_i\leq K_i$;
\item $[[K_1,\dots,K_{i}],K_{i+1},\dots ,K_n ]\leq [K_1,\dots,K_i,K_{i+1},\dots,K_n]$;
\item $[K_1,\dots,K_{i-1},K_i,K_{i+1},\dots ,K_n ]\leq [K_1,\dots,K_{i-1},K_{i+1},\dots,K_n]$ whenever $K_i=K_{i+1}$;
\item $[K_1,\dots,K_{n-1},K_n\vee L_n]=[K_1,\dots,K_{n-1},K_n]\vee [K_1,\dots,K_{n-1},L_n]\vee [K_1,\dots,K_{n-1},K_n,L_n]$.\noproof
\end{enumerate}
\end{proposition}

One concrete application of a higher Higgins commutator is in the expression of the \emph{Smith is Huq} condition \SH\ mentioned above. Let $R$ and $S$ be equivalence relations on an object $X$, and let $K$ and $L\normal X$ denote their normalisations (=~zero-classes). Theorem~4.16 in~\cite{HVdL} says that the normalisation of the Smith/Pedicchio commutator~\cite{Pedicchio, Borceux-Bourn} of~$R$ and~$S$ is the normal subobject $[K,L]\join [K,L,X]$ of $X$. As a consequence, the condition \SH\ holds if and only if this join is a subobject of the Huq commutator $[K,L]\join [[K,L],X]$ of $K$ and $L$. Hence the category $\CC$ satisfies \emph{both} conditions \SH\ and \NH\ if and only if $[K,L,X]\leq [K,L]$ whenever $K$, $L\normal X$: this is Proposition~6.1 in~\cite{CGrayVdL1}. We now describe a convenient class of categories having this property, and (as it turns out) even satisfying a higher-order version of it.

\subsection{Algebraically coherent categories}\label{Fibration of points}
The category $\Pt(\CC)$ of \defn{points in $\CC$} has split epimorphisms with a chosen splitting, so pairs $(p\colon {Z\to X},s\colon {X\to Z})$ where $p\comp s =1_X$, as objects and natural transformations between those as morphisms. The fibre over an object $X$ of $\CC$ is written $\Pt_X(\CC)$; a morphism from $(p,s)$ to $(p'\colon {Z'\to X},s'\colon {X\to Z'})$ is a map $z\colon {Z\to Z'}$ in $\CC$ that satisfies $z\comp s=s'$ and $p'\comp z=p$.

Given any morphism $f\colon X'\to X$ we may pull back or push out along it, and thus obtain the change-of-base functors 
\[
f^*\colon \Pt_X(\CC)\to \Pt_{X'}(\CC)
\qquad\text{and}\qquad
f_*\colon \Pt_{X'}(\CC)\to \Pt_{X}(\CC),
\]
where $f_{*}\dashv f^*$.
Recall that Bourn-protomodularity is equivalent to the condition that the functors~$f^*$ are conservative. A semi-abelian category~$\CC$ is said to be \defn{algebraically coherent}~\cite{acc} when moreover the functors $f^*$ are \defn{coherent}, which means that they preserve finite limits and jointly extremally epimorphic pairs of arrows.

In the special case where $f=\gnab_X\colon{0\to X}$ for some $X$ in $\CC$, we find the functor
\[
\gnab_{X}^{*}\colon \Pt_X(\CC)\to \CC\colon (p,s)\mapsto \Ker(p)
\]
and its left adjoint 
\[
(\gnab_{X})_{*}\colon \CC\to \Pt_X(\CC)\colon Y \mapsto (\mus{0}{1_X}\colon Y+X\to X,\;\iota_X\colon X\to Y+X).
\]
Given an object $Y$, the process of first applying the left adjoint $(\gnab_{X})_{*}$ to it, then the right adjoint $\gnab_{X}^{*}$ to the result yields an object $X\flat Y$, which is the kernel in the short exact sequence
\[
\xymatrix{0 \ar[r] & X\flat Y \ar@{{ |>}->}[r]^{\kappa_{X,Y}} & Y+X \ar@{-{ >>}}@<.5ex>[r]^-{\mus{0}{1_X}} & X \ar@{{ >}->}@<.5ex>[l]^-{\iota_X} \ar[r] & 0.}
\]
The adjunction $(\gnab_{X})_{*}\dashv\gnab_{X}^*$ is monadic: the functor $X\flat-\colon \CC\to \CC$ is part of a monad, whose algebras are the \emph{internal actions} of $X$; via a semidirect product construction~\cite{Bourn-Janelidze:Semidirect, BJK}, the category of $X$-actions in $\CC$ is equivalent to $\Pt_X(\CC)$. Here we only need the monad's unit: the inclusion $\iota_Y\colon {Y\to Y+X}$ factors over the kernel $\kappa_{X,Y}$ as a split monomorphism $\eta_Y^X\colon {Y\to X\flat Y}$ with splitting $\tau^X_Y\coloneq \lgroup 1_Y\;0\rgroup\comp \kappa_{X,Y}$.

\begin{proposition}\cite{acc}\label{Proposition AC}
For a semi-abelian category $\CC$, the following conditions are equivalent:
\begin{tfae}
	\item $\CC$ is algebraically coherent;
	\item the change-of-base functors $\gnab_{X}^{*}\colon \Pt_X(\CC)\to \CC$ are coherent;
	\item the natural comparison morphism $\mus{X\flat \iota_Y}{X\flat \iota_Z}\colon X\flat Y+ X\flat Z \to X\flat (Y+Z)$ is a regular epimorphism, for each choice of $X$, $Y$, $Z\in \CC$.\noproof
\end{tfae}
\end{proposition}

All \emph{locally algebraically cartesian closed} semi-abelian categories~\cite{Gray2012} are examples, since then the comparison morphisms of condition (iii) are isomorphisms. We find groups, Lie algebras, crossed modules, cocommutative Hopf algebras over a field of characteristic zero. Next we have all \emph{Orzech categories of interest}. In the case of non-associative algebras, we find that algebraic coherence is equivalent to the conditions of Theorem~\ref{Theorem Algebras}.

All algebraically coherent semi-abelian categories satisfy both \SH\ and \NH. More precisely, we have the following:
\begin{proposition}
For $\CC$ semi-abelian, the following conditions are equivalent:
\begin{tfae}
\item the change-of-base functors $\gnab_{X}^{*}\colon \Pt_X(\CC)\to \CC$ preserve Huq commutators of pairs of normal subobjects;
\item $\CC$ satisfies \SH$+$\NH;
\item $[K,L,X]\leq [K,L]$ whenever $K$, $L\normal X$.
\end{tfae}
Furthermore, these conditions hold when $\CC$ is algebraically coherent.
\end{proposition}
\begin{proof}
The equivalence between (i) and (ii) is part of Theorem~6.5 in~\cite{CGrayVdL1}. The equivalence between (ii) and (iii) is Proposition~6.1 in~\cite{CGrayVdL1}; let us sketch its proof, freely using Proposition~\ref{Normality via commutator}. If (iii) holds then 
\[
[[K,L],X]\leq [K,L,X]\leq [K,L],
\]
so that $[K,L]\normal X$ (the condition \NH\ holds) while
\[
[K,L,X]\leq[K,L]\leq [K,L]\join [[K,L],X]
\]
(the condition \SH\ holds). Conversely, under \SH, we have $[K,L,X]\leq [K,L]\join [[K,L],X]$; now \NH\ implies that $[[K,L],X]\leq [K,L]$, so that condition (iii) holds.

It remains to be shown that any coherent functor preserves Huq commutators of pairs of normal subobjects. The reason is that coherent functors preserve Higgins commutators (by Proposition~6.9 in~\cite{acc}, generalised to Proposition~\ref{Proposition Coherent Functor Preserve Higgins Commutators} below) and binary joins of subobjects (essentially by definition); hence normal closures are preserved as well, which entails preservation of Huq commutators. 
\end{proof}

The fact that algebraic coherence implies condition (iii) may also be viewed as a special case of Theorem~\ref{Theorem Inequality} below, which generalises this condition to an arbitrary finite number of normal subobjects $K_i\normal X$.

\subsection{Structure of the text}
Section~\ref{Section Preliminaries} gives an overview of notations and basic results having to do with cubic extensions and cross-effects. This is used in Section~\ref{Section A Commutator Inequality} where we prove a key technical result: Theorem~\ref{Theorem Inequality}, which says that for $n\geq 2$, given $n$ normal subobjects $K_1$, \dots, $K_n$ of an object $X$ in an algebraically coherent semi-abelian category, the $(n+1)$-fold commutator $[K_1, \dots,K_n,X]$ is contained in the $n$-fold commutator $[K_1, \dots,K_n]$. This result is crucial in Section~\ref{Section n Subobjects Lemma}, where it is used in the proof that higher commutators decompose into joins of nested lower-order commutators: the ``$n$ Subobjects Lemma'', Theorem~\ref{Theorem n Subobjects Lemma}. This easily leads to Corollary~\ref{Corollary Nilpotency} saying that the two above-mentioned types of nilpotency coincide, and to Theorem~\ref{Theorem Decomposition into binary commutators}, which provides a decomposition of any higher commutator into a join of nested binary commutators.

\section{Preliminaries on extensions and cross-effects}\label{Section Preliminaries}

The concept of an \emph{$n$-cubic extension} first occurred in the approach to homology via categorical Galois theory~\cite{Donadze-Inassaridze-Porter, EGVdL}; closely related to this is the fact that \emph{central} $n$-cubic extensions are classified by the higher cohomology groups~\cite{RVdL2}. We need cubic extensions here because they allow an alternative description of a co-smash product and, more generally, a cross-effect as the so-called \emph{direction} of a certain $n$-cubic extension. This allows us to deduce certain information about those cross-effects or co-smash products.

\subsection{Extensions}
We first recall some definitions and properties from~\cite{EGVdL, EGoeVdL}. For $n> 0$ we consider the set $\bar n\coloneq\{1,\dots,n\}$. By an \defn{$n$-fold arrow} $A$ in $\CC$ we mean a contravariant functor 
\[
A\colon \PP(\bar n)^{\op}\to \CC,
\]
where $\PP(\bar n)$ denotes the powerset of $\bar n$. $0$-fold arrows are objects of $\CC$. A morphism between $n$-fold arrows $A$ and $B$ is a natural transformation $f\colon A\To B$. We write $\Arr^0(\CC)\coloneq \CC$, and when $n > 0$, $\Arr^n(\CC)\coloneq\Fun(\PP(\bar n)^{\op}, \CC)$ for the category of $n$-fold arrows and morphisms between them. If $A$ is an $n$-fold arrow and $I\subseteq \bar n$, then $A(I)$ denotes the image of $I$ by the functor $A$. Thus for simplicity, sometimes, the $n$-fold arrow $A$ is written as an \defn{$n$-cube} $A\coloneq(A(I))_{I\subseteq \bar n}$ in $\CC$.

An $n$-fold arrow $E$ is an \defn{$n$-cubic extension} when for all $\emptyset\neq I\subseteq \bar n$ the arrow $E(I) \to \lim_{J\subsetneq I}E(J)$ is a regular epimorphism. (The limit $\lim_{J\subsetneq \bar n}E(J)$ of ``the cube $E$ minus its initial object $E(\bar n)$'' merits special attention and a notation of its own: we write it $\L(E)$, and the induced regular epic comparison arrow is denoted $\lambda_E\colon E(\bar n)\to \L(E)$.) We write $ \Ext^n(\CC) $ for the category of $n$-cubic extensions and morphisms between them. It is a full subcategory of $\Arr^n(\CC)$. Equivalently, an $n$-cubic extension in $\CC$ is a commutative square of solid arrows 
\[
\vcenter{\xymatrix@!0 @R=4em @C=5em{A_1\ar@/^/[rrd]\ar@/_/[rdd] \ar@{.>}[rd] &&\\
		&P \ar@{.>}[r] \ar@{.>}[d] \dottedpullback &B_1\ar[d]\\
		&A_0\ar[r]&B_0}}
\]
in $\Arr^{n-2}(\CC)$ such that all arrows in the diagram are $(n-1)$-cubic extensions. 

We use interchangeably ``regular epimorphism'', ``extension'' and ``$1$-cubic extension'', and ``double extension'' means ``$2$-cubic extension''.

\begin{lemma}\label{Lemma Split Epimorphism of Extensions}
In a regular Mal'tsev category, any split epimorphism between $n$-cubic extensions is an $(n+1)$-cubic extension.
\end{lemma}
\begin{proof}
	This follows by induction from Lemma~3.2 in~\cite{EGoeVdL} via its Example~3.14.
\end{proof}

A $1$-fold arrow is a \defn{split $1$-cubic extension} when it is a split epimorphism. By induction, an $(n+1)$-fold arrow is a \defn{split $(n+1)$-cubic extension} (an ``$(n+1)$-fold split epimorphism'') when it is a split epimorphism of split $n$-cubic extension. By the above lemma it is indeed an extension; this may be generalised as follows~\cite{EGoeVdL}, to a property which for $n=1$ characterises Mal'tsev categories among regular categories, as shown by Bourn in~\cite{Bourn1996}, cf.\ \cite[Corollary~1.7]{BeBou}. 

\begin{lemma}\label{Lemma Regular Epimorphism of Split Cubes}
In a regular Mal'tsev category, any regular epimorphism between split $n$-cubic extensions is an $(n+1)$-cubic extension.
\end{lemma}
\begin{proof}
	We may view any regular epimorphism between split $n$-cubic extensions as a split epimorphism between $n$-cubic extensions. The result now follows from Lemma~\ref{Lemma Split Epimorphism of Extensions}.
\end{proof}

\begin{lemma}\label{Lemma induced double extension}
In a pointed regular category $\CC$, let $A$ be an $n$-fold arrow, with $n>1$. View $A$ as a morphism $A\colon{\dom(A)\to\cod(A)}$ in $\Arr^{n-1}(\CC)$ between the $(n-1)$-fold arrows $\dom(A)$ and $\cod(A)$. The induced diagram 
\[
\xymatrix@C=4em@R=2em{\dom(A)(\overline{n-1})\ar[r]^-{\lambda_{\dom(A)}}\ar[d]&\L(\dom(A))\ar[d]\\
\cod(A)(\overline{n-1})\ar[r]_-{\lambda_{\cod(A)}}&\L(\cod(A))}
\]
is a $2$-fold arrow $B$ in $\CC$, for which there exists an isomorphism $\tau\colon {\L(A)\to \L(B)}$ such that $\lambda_B=\tau\comp\lambda_A$. If $A$ is an $n$-cubic extension, then $B$ is a double extension in $\CC$.
\end{lemma}
\begin{proof}
The isomorphism $\tau$ compares two constructions of the limit $\lim_{J\subsetneq \bar n}A(J)$: the given $\L(A)$ on the one hand, and the pullback induced by the square $B$ on the other. If now $A$ is an $n$-cubic extension, then $\lambda_A$, and thus also $\lambda_B$, is a regular epimorphism. Moreover, the domain and codomain of $A$ are $(n-1)$-cubic extensions, so that $\lambda_{\dom(A)}$ and $\lambda_{\cod(A)}$ are regular epimorphisms as well. Since also the vertical arrow on the left of the square $B$ is a regular epimorphism, $B$ is a double extension in $\CC$.
\end{proof}

We now focus on a special type of $n$-cubic extensions: those obtained out of coproducts of a given finite collection of objects.

\subsection{Split extensions obtained via coproducts}\label{Extensions via coproducts}
We write $E(X_1, \dots,X_n)$ for the $n$-cubic extension defined by
\[
E(X_1, \dots,X_n)(I)=\coprod_{i\in I}X_i
\]
for $I\subseteq \bar n$---in particular, $E(X_1, \dots,X_n)(\emptyset)=0$---and
\[
\widehat \pi_{X_i}\colon E(X_1, \dots,X_n)(I\cup \{i\})=\coprod_{k\in I\cup \{i\}}X_k\to E(X_1, \dots,X_n)(I)=\coprod_{k\in I}X_k
\]
whenever $i\in \bar n\setminus I$, which sends $X_i$ to zero and $X_k$ for $k\in I$ to itself via $1_{X_k}$. Note that it is a split $n$-cubic extension. A canonical splitting $\widehat \iota_{X_i}$ of $\widehat \pi_{X_i}$ is 
\[
\iota_{\coprod_{k\in I} X_k}\colon E(X_1, \dots,X_n)(I)=\coprod_{k\in I}X_k\to E(X_1, \dots,X_n)(I\cup \{i\})=\coprod_{k\in I\cup \{i\}}X_k.
\]

\subsection{Cross-effects}\label{Subsection Cross-Effects}\label{Subsection Cosmash}
A co-smash product is a special instance of a \emph{cross-effect}; we recall definitions and properties from~\cite{HVdL, BeBou}. Let $F\colon \CC\to \DD$ be a functor from a pointed category with finite sums $\CC$ to a pointed finitely 
complete category $\DD$. The \textbf{$n$\textsuperscript{th} cross-effect} of $F$ is the functor 
\[
\Cr_n(F)\colon \CC^n\to \DD
\]
defined by
\begin{align*}
\Cr_1(F)(X)&\coloneq\Ker\bigl(F(0)\colon F(X)\to F(0)\bigr)\\
\Cr_n(F)(X_1, \dots,X_n)&\coloneq\Ker\bigl(r_{X_1,\dots,X_n}^F\bigr)
\end{align*}
where 
\[
r_{X_1, \dots,X_n}^F\coloneq \musvv{F(\widehat\pi_1)}{F(\widehat\pi_n)}\colon F(X_1+\cdots+X_n)\to \prod_{k=1}^nF(X_1+\cdots+\widehat X_k+\cdots+X_n).
\]
Equivalently, $\Cr_n(F)(X_1, \dots,X_n)$ is the kernel of the morphism 
\[
\lambda^F_{X_1, \dots,X_n}\colon F\bigl(E(X_1, \dots,X_n)\bigr)(n) \to \L\bigl(F(E(X_1, \dots,X_n))\bigr)
\]
where $\L\bigl(FE(X_1, \dots,X_n)\bigr)$ is the limit of the diagram $F\bigl(E(X_1, \dots,X_n)\bigr)$ restricted to the category $(\PP(\bar n)\setminus \{\bar n\})^{\op}$ and $\lambda^F_{X_1, \dots,X_n}$ is the universally induced comparison morphism. Indeed, $r_{X_1, \dots,X_n}^F$ can be written as $\lambda^F_{X_1, \dots,X_n}$ followed by a monomorphism. When $\DD$ is semi-abelian, $\lambda^F_{X_1, \dots,X_n}$ is a regular epimorphism. 

We also write $F(X_1|\cdots|X_n)$ for $\Cr_n(F)(X_1, \dots,X_n)$. In particular, when $F$ is the identity functor $1_{\CC}$, we obtain
\[
X_1\cosmash\cdots\cosmash X_n\coloneq 1_{\CC}(X_1|\cdots|X_n)=\Cr_n(1_{\CC})(X_1, \dots,X_n),
\]
the co-smash product of $X_1$, \dots, $X_n$ as in~\ref{Higher order}. We find a short exact sequence
\[
\xymatrix@=3em{0 \ar[r] & X_1\cosmash \cdots \cosmash X_n \ar@{{ |>}->}[r]^-{\iota_{X_1, \dots,X_n}} & X_1+ \cdots + X_n \ar@{-{ >>}}[r]^-{\lambda_{X_1, \dots,X_n}} & \L E(X_1, \dots,X_n) \ar[r] & 0}
\]
in the semi-abelian category $\CC$. Notice the \emph{absence} of brackets here: any bracketing may result in a different object.

\subsection{The direction of a higher arrow}\cite{RVdL2}
The \defn{direction} $\D_n(A)$ of an $n$-fold arrow $A$ is the kernel in $\CC$ of the comparison morphism $\lambda_A\colon A(\bar n)\to \L (A)$, universally induced by the property of $\L (A)=\lim_{J\subsetneq \bar n}A(J)$. This defines a functor
\[
\D_n\colon\Arr^n(\CC)\to \CC\colon A\mapsto \D_n(A).
\]
For instance, the co-smash product $X_1\cosmash \cdots \cosmash X_n$ is the direction $\D_n(E(X_1, \dots,X_n))$ of the $n$-cube $E(X_1, \dots,X_n)$. Since limits commute with limits, it is clear that the functor $\D_n$ preserves all kernels. In fact, when we restrict its domain to $\Ext^n(\CC)$, it also preserves extensions:

\begin{lemma}\label{Lemma D preserves extensions}
In a pointed regular category $\CC$, for each $n\geq 1$ the direction functor $\D_n\colon\Ext^n(\CC)\to \CC\colon E\mapsto \D_n(E)$ preserves extensions.
\end{lemma}
\begin{proof}
For $n=1$ the result follows immediately. If now $E$ is an $n$-cubic extension in $\CC$, with $n>1$, then Lemma~\ref{Lemma induced double extension} tells us that the right-hand side square in the diagram
\[
\xymatrix@C=4em@R=2em{0 \ar[r] & \D_{n-1}(\dom(E))\ar@{-{ >>}}[d]_-{\D_{n-1}(E)}\ar@{{ |>}->}[r]^-{\ker(\lambda_{\dom(E)})}&\dom(E)(\overline{n-1})\ar@{-{ >>}}[r]^-{\lambda_{\dom(E)}}\ar@{-{ >>}}[d]&\L(\dom(E))\ar@{-{ >>}}[d]\\
0 \ar[r] &  \D_{n-1}(\cod(E))	\ar@{{ |>}->}[r]_-{\ker(\lambda_{\cod(E)})}&\cod(E)(\overline{n-1})\ar@{-{ >>}}[r]_-{\lambda_{\cod(E)}}&\L(\cod(E))}
\]
is a double extension in $\CC$. Hence $\D_{n-1}(E)$ is a regular epimorphism, so an extension, in~$\CC$.
\end{proof}

Consider objects $K_1$, \dots, $K_n$ and $X$ in $\CC$. Then we may view the $(n+1)$-cubic extension $E(K_1, \dots,K_n,X)$ as a morphism
\[
E_X(K_1,\dots,K_n)\to E(K_1,\dots,K_n)
\]
in $\Arr^{n}(\CC)$, where the domain $n$-cubic extension $E_X(K_1,\dots,K_n)$ is the subdiagram determined by the coproducts of the form $X+\coprod_{i\in I} K_i$. In particular, $C_X=E_X(0,\dots,0)$ is the constant functor $\PP(\bar n)^{\op}\to \CC$ with value $X$.

 \begin{lemma}\label{Fundamental lemma for directions} 
Consider objects $K_1$, \dots, $K_n$ and $X$ in a pointed regular category~$\CC$. The functor ${\gnab_{X}^{*}\colon \Pt_X(\CC)\to \CC}$ sends the direction
\[
(\gnab_{X})_{*}(K_1)\cosmash_X\cdots \cosmash_X (\gnab_{X})_{*}(K_n)
\]
of the $n$-cube $E\coloneq E((\gnab_{X})_{*}(K_1), \dots, (\gnab_{X})_{*}(K_n))$ in $\Pt_X(\CC)$ to the direction of the $n$-cube $E_X\coloneq E_X(K_1,\dots,K_n)$ in $\CC$.
\end{lemma}
\begin{proof} 
The $n$-cube $E$ in $\Pt_X(\CC)$ may be viewed as a split epimorphism $e\colon E_X\to C_X$ of $n$-cubes in $\CC$. By construction, its kernel $K$ is the image of $E$ through the functor ${\gnab_{X}^{*}\colon \Pt_X(\CC)\to \CC}$. Since the direction functor $\D_n$ preserves kernels, we have that $\D_n(\gnab_{X}^{*}(E))=\D_n(K)$ is the kernel of $\D_n(e)\colon \D_n(E_X)\to \D_n(C_X)=0$. Hence $\D_n(E_X)\cong \D_n(\gnab_{X}^{*}(E))$. Since kernels commute with limits, the latter direction $\D_n(\gnab_{X}^{*}(E))$ is nothing but the image through $\gnab_{X}^{*}$ of the direction of $E$ in $\Pt_X(\CC)$. This proves our claim.
\end{proof}

\section{A commutator inequality}\label{Section A Commutator Inequality}

The aim of this section is to prove a key technical result: Theorem~\ref{Theorem Inequality}, which says that for $n\geq 2$, given $n$ normal subobjects $K_1$, \dots, $K_n$ of an object $X$ in an algebraically coherent semi-abelian category, the $(n+1)$-fold commutator $[K_1, \dots,K_n,X]$ is contained in the $n$-fold commutator $[K_1, \dots,K_n]$. This generalises---see the paragraph preceding~\ref{Fibration of points}---the condition \SH+\NH, which by~\cite[Proposition~6.1]{CGrayVdL1} may be seen as the special case where $n=2$, and will turn out to be crucial for the proofs in the next section.

We start with Proposition~\ref{Proposition Coherent Functor Preserve Higgins Commutators}, a generalisation of Proposition~6.9 in~\cite{acc} which says that coherent functors preserve binary Higgins commutators. 

Given objects $X_1$, \dots, $X_n$ in a pointed regular category with binary coproducts, write $\iota_i\colon X_i\to X_1+\cdots+X_n$ for the canonical inclusion. Note that this is a jointly extremally epimorphic family~\cite[Proposition~A.4.18]{Borceux-Bourn}. The next lemma then follows immediately from the definition of a coherent functor, which preserves finite limits and finite jointly extremally epimorphic families of arrows.

\begin{lemma}\label{Lemma Coherent Comparison}
Let $F\colon {\CC\to \DD}$ be a coherent functor between pointed regular categories with binary coproducts. Then for any $X_1$, \dots, $X_n$ in $\CC$, the arrow
\[
\lgroup F(\iota_1)\;\cdots\;F(\iota_n)\rgroup \colon F(X_1)+\cdots + F(X_n)\to F(X_1+\cdots+X_n)
\] 
is a regular epimorphism in $\DD$.
\noproof
\end{lemma}

\begin{proposition}\label{Proposition Coherent Functor Preserve Higgins Commutators}
Let $F\colon {\CC\to \DD}$ be a coherent functor between pointed regular Mal'tsev categories with binary coproducts. Then $F$ preserves Higgins commutators (of arbitrary length).
\end{proposition}
\begin{proof}
Consider $n\geq 2$ and let $K_1$, \dots, $K_n$ be subobjects of an object $X$ in $\CC$, each represented by a monomorphism denoted $k_i\colon K_i\to X$. Since the $n$-cubic extension $E(K_1, \dots,K_n)$ from~\ref{Extensions via coproducts} is a split $n$-cubic extension, so is the $n$-cube $F(E(K_1, \dots,K_n))$. By Lemma~\ref{Lemma Coherent Comparison}, the canonical comparison morphism 
\[
r\colon E(F(K_1), \dots,F(K_n))\to F(E(K_1, \dots,K_n))
\]
is a regular epimorphism between split $n$-cubic extensions. Note that the components of $r$ do indeed commute with the compatible splittings of~\ref{Extensions via coproducts}. All faces of $r$ are regular epimorphisms of split epimorphisms, since for all $I\subsetneq \bar n$ and any $i\in \bar n\setminus I$, both squares in
\[
\xymatrix@C=5em{E(F(X_1), \dots, F(X_n))(I\cup \{i\}) \ar@<-.5ex>[d]_-{\widehat \pi_{F(X_i)}} \ar@{-{ >>}}[r]^-{(F(\iota_k))_{k\in I\cup \{i\}}} & F(E(X_1, \dots, X_n))(I\cup \{i\}) \ar@<-.5ex>[d]_-{F(\widehat \pi_{X_i})}\\
E(F(X_1), \dots, F(X_n))(I) \ar@<-.5ex>[u]_-{\widehat \iota_{F(X_i)}} \ar@{-{ >>}}[r]_-{(F(\iota_k))_{k\in I}} & F(E(X_1, \dots, X_n))(I) \ar@<-.5ex>[u]_-{F(\widehat \iota_{X_i})}}
\]
commute. Hence by Lemma~\ref{Lemma Regular Epimorphism of Split Cubes}, the $(n+1)$-cube $r$ is an extension. Since the coherent functor $F$ preserves all limits, the direction of $F(E(K_1, \dots,K_n))$ is the object $F(K_1\cosmash \cdots \cosmash K_n)$. By Lemma~\ref{Lemma D preserves extensions}, the morphism 
\[
\D_{n}(r)\colon F(K_1)\cosmash \cdots \cosmash F(K_n)\to F(K_1\cosmash \cdots \cosmash K_n)
\]
which we find when taking directions is a regular epimorphism. Since the coherent functor~$F$ is regular, it preserves image factorisations. Hence $F([K_1, \dots,K_n])$ is the image of the morphism
\[
F\bigl(\xymatrix@1@C=5em{K_1\cosmash \cdots \cosmash K_n \ar[r]^-{\iota_{K_1, \dots,K_n}} & K_1+ \cdots + K_n \ar[r]^-{\lgroup k_1\;\cdots\;k_n\rgroup} & X}\bigr),
\]
which is also the image of
\[
F(\lgroup k_1\,\cdots\,k_n\rgroup)\comp F(\iota_{K_1, \dots,K_n})\comp \D_{n}(r)=\lgroup F(k_1)\,\cdots\,F(k_n)\rgroup\comp \iota_{F(K_1), \dots,F(K_n)}, 
\]
which is nothing but the commutator $[F(K_1), \dots,F(K_n)]$.
\end{proof}

If $\CC$ is a semi-abelian algebraically coherent category, then for any object $X$ of~$\CC$ we may apply this result to the functor $\gnab_{X}^{*}\colon \Pt_X(\CC)\to \CC\colon (p,s)\mapsto \Ker(p)$, and thus we obtain the following higher-order version of the condition \SH+\NH, which by Theorem~4.6 in~\cite{HVdL} and Proposition~6.1 in~\cite{CGrayVdL1} may be seen as the special case where $n=2$:

\begin{theorem}\label{Theorem Inequality}
In a semi-abelian algebraically coherent category, consider $n\geq 2$, and $n$ subobjects $K_1$, \dots, $K_n$ of an object $X$. Write $\overline K_i\normal X$ for the normal closure of $K_i\leq X$. Then 
\begin{equation}\label{Inequality}
[K_1, \dots,K_n,X]\leq [\overline K_1, \dots,\overline K_n].
\end{equation}
In particular, if $K_1$, \dots, $K_n\normal X$, then
\[
[K_1, \dots,K_n,X]\leq [K_1, \dots,K_n].
\]
\end{theorem}
\begin{proof}
On the one hand, the co-smash product $K_1\cosmash \cdots \cosmash K_n\cosmash X$ is the direction of the $(n+1)$-cubic extension $E(K_1, \dots,K_n,X)$. The commutator on the left hand side of the inequality~\eqref{Inequality} is the image of the composite morphism
\[
\xymatrix@C=5em{K_1\cosmash \cdots \cosmash K_n\cosmash X \ar[r]^-{\iota_{K_1, \dots,K_n,X}} & K_1+ \cdots + K_n+ X \ar[r]^-{\lgroup k_1\,\cdots\,k_n\,1_X\rgroup} & X.}
\]
On the other hand, for each $1\leq i\leq n$ we may consider the point
\[
(\gnab_{X})_{*}(K_i)\defeq(\mus{0}{1_X}\colon K_i+X\to X,\; \iota_{X}\colon X\to K_i+X).
\]
The morphism 
\[
k'_i\defeq\bigl\lgroup\begin{smallmatrix}
k_i & 1_X\\
0 & 1_X
\end{smallmatrix}\bigr\rgroup
\colon K_i+X\to X\times X
\]
in $\CC$ may be viewed as a morphism of points with codomain
\[
(\pi_2\colon {X\times X\to X},\; \musv{1_X}{1_X}\colon X\to X\times X).
\]
Its image is a subobject $(p_i\colon K'_i\to X,\; s_i\colon X\to K'_i)$ of the point $(\pi_2,\musv{1_X}{1_X})$. These $(p_i,s_i)\leq(\pi_2,\musv{1_X}{1_X})$ have a Higgins commutator in $\Pt_X(\CC)$, which coincides with the image of the composite arrow
\[
\resizebox{\textwidth}{!}{\xymatrix@=4em{(\gnab_{X})_{*}(K_1)\cosmash_X\cdots \cosmash_X (\gnab_{X})_{*}(K_n) \ar[rr]^-{\iota_{(\gnab_{X})_{*}(K_1), \dots,(\gnab_{X})_{*}(K_n)}} && (\gnab_{X})_{*}(K_1)+_X\cdots +_X (\gnab_{X})_{*}(K_n) \ar[r]^-{\lgroup k_1' \;\cdots\; k_n' \rgroup} & X\times X,}}
\]
considered as a morphism $\theta$ in $\Pt_X(\CC)$. Here the first arrow is the canonical inclusion, and the second arrow is induced by the $k'_i\colon K_i+X\to X\times X$. Since coherent functors preserve image factorisations, Proposition~\ref{Proposition Coherent Functor Preserve Higgins Commutators} tells us that the change-of-base functor $\gnab_{X}^{*}\colon \Pt_X(\CC)\to \CC$ sends this commutator to the image of the composite 
\[
\xymatrix{(X\flat K_1)\cosmash\cdots \cosmash (X\flat K_n) \ar[r] & (X\flat K_1)+\cdots + (X\flat K_n) \ar[r] & X,}
\]
where the first arrow is the canonical inclusion, and the second arrow is induced by the $\lgroup k_i\; 1_X\rgroup\comp \kappa_{X,K_i}\colon X\flat K_i\to K_i+ X\to X$. The image of this latter morphism being the normal closure $\overline K_i$ of $K_i$ in $X$---see~\cite{MM-NC},~\cite{MFVdL4}, or~\cite{acc}---the  functor $\gnab_{X}^{*}\colon \Pt_X(\CC)\to \CC$ sends the given commutator in $\Pt_X(\CC)$ to the commutator $[\overline K_1, \dots,\overline K_n]$ on the right hand side of the inequality~\eqref{Inequality}. Further, note that the canonical comparison arrow
\[
\overline{\rho}\colon(X\flat K_1)\cosmash\cdots \cosmash (X\flat K_n)\to \gnab_X^*\bigl((\gnab_{X})_{*}(K_1)\cosmash_X\cdots \cosmash_X (\gnab_{X})_{*}(K_n)\bigr)
\]
is a regular epimorphism by algebraic coherence. Indeed, the functor $X\flat - :\CC\to \CC$ is coherent and $ \gnab_X^*\bigl((\gnab_{X})_{*}(K_1)+_X\cdots+_X (\gnab_{X})_{*}(K_n)\bigr) =X\flat\bigl( E(K_1,\dots,E_n)\bigr)$. The claim holds because the $(n+1)$-cube
\[
E(X\flat K_1,\dots,X\flat K_n)\to X\flat\bigl (E(K_1,\dots,E_n)\bigr)
\]
is an $(n+1)$-extension, so that in the diagram
\[
\xymatrix@!0@R=4em @C=17em{ (X\flat K_1)\cosmash\cdots \cosmash(X\flat K_n)\ar@{{ |>}-{>}}[d]_-{\iota_{X\flat K_1,\dots,X\flat K_n}}\ar@{-{ >>}}[r]^-{\bar \rho}& \gnab_X^*\bigl((\gnab_{X})_{*}(K_1)\cosmash_X\cdots \cosmash_X (\gnab_{X})_{*}(K_n)\bigr)\ar@{{ |>}-{>}}[d]^-{\gnab^{*}_{X}(\iota_{(\gnab_{X})_{*}(K_1),\dots,(\gnab_{X})_{*}(K_n)})}\\
	X\flat K_1+\cdots +X\flat K_n\ar@{-{ >>}}[r]^-{\rho=\lgroup 1_X\flat \iota_{1}\;\cdots\;1_X\flat \iota_{n}\rgroup} \ar@{-{ >>}}[d]_-{\lambda_{X\flat K_1,\dots,X\flat K_n}}	& X\flat(K_1+\cdots +K_n)\ar@{-{ >>}}[d]^-{\gnab^{*}_X(\lambda_{(\gnab_{X})_{*}(K_1),\dots,(\gnab_{X})_{*}(K_n)})} \\
	\L(E(X\flat K_1,\dots,X\flat K_n)\ar@{-{ >>}}[r]	&\L\bigl(X\flat (E(K_1,\dots,K_n))\bigr)}
\]
the bottom square is a double extension, and then the top morphism $\bar \rho $ is a regular epimorphism.

Furthermore, the restriction of the limit cone
\[
\L E(K_1, \dots,K_n,X)\Rightarrow E(K_1, \dots,K_n,X)|_{(\PP(\bar n)\setminus \{\bar n\})^{\op}}
\] 
to the $n$-cube $E_X(K_1,\dots, K_n)$ is still a cone over $E_X(K_1,\dots, K_n)$. Thus, via the universal property of the limit $\L\bigl(E_X(K_1,\dots,K_n)\bigr)$ and by Lemma~\ref{Fundamental lemma for directions}, we find the dotted arrow on the right in the diagram
\[
\resizebox{\textwidth}{!}{
	\xymatrix@C=2em@R=4em{0 \ar[r] & K_1\cosmash \cdots \cosmash K_n \cosmash X \ar@{.>}[d]_{\xi} \ar@{{ |>}->}[rr]^-{\iota_{X_1, \dots,X_n}} && K_1+ \cdots + K_n + X \ar@{=}[d] \ar@{-{ >>}}[rr]^-{\lambda_{X_1, \dots,X_n}} && \L E(K_1, \dots,K_n,X) \ar@{.>}[d] \ar[r] & 0\\
		0 \ar[r] & \gnab_X^*\bigl((\gnab_{X})_{*}(K_1)\cosmash_X\cdots \cosmash_X (\gnab_{X})_{*}(K_n)\bigr) \ar@{{ |>}->}[rr]_-{\ker\bigl(\lambda_{E_X(K_1,\dots,K_n)}\bigr)} && K_1+ \cdots + K_n + X \ar@{-{ >>}}[rr]_-{\lambda_{E_X(K_1,\dots,K_n)}} && \L\bigl(E_X(K_1,\dots,K_n)\bigr) \ar[r] & 0,}}
\]
which displays a morphism of short exact sequences. Note that the bottom sequence is exact, because the kernel in $\CC$ of an arrow in~$\Pt_X(\CC)$ coincides with the kernel in~$\CC$ of its image through $\gnab_X^*$. Via the universal property of strong epimorphisms, the dotted arrow $\xi$ on the left induces the inequality~\eqref{Inequality}, through the diagram
\[
\xymatrix@!0@C=10em@R=4em{K_1\cosmash\cdots \cosmash K_n\cosmash X \ar[rd]^-{\xi} \ar@{-{ >>}}[dd] && (X\flat K_1)\cosmash \cdots \cosmash (X\flat K_n) \ar@{-{ >>}}[ld]_-{\overline{\rho}} \ar@{-{ >>}}[dd] \\
	& \gnab_X^*\bigl((\gnab_{X})_{*}(K_1)\cosmash_X\cdots \cosmash_X (\gnab_{X})_{*}(K_n)\bigr) \ar@{-{ >>}}[rd] \ar[dd]_(.25){\gnab_X^*(\theta)}\\
	[K_1, \dots,K_n,X] \ar@{{ >}.>}[rr] \ar@{{ >}->}[rd] && [\overline K_1, \dots,\overline K_n] \ar@{{ >}->}[ld] \\
	& \gnab_X^*(X\times X)=X}
\]
This completes the proof.
\end{proof}

\section{The \texorpdfstring{``$n$ Subobjects Lemma''}{n Subobjects Lemma}}\label{Section n Subobjects Lemma}
In this section we extend Theorem~\ref{Three Subobjects Lemma}---the \emph{Three Subobjects Lemma} of~\cite{acc}, valid in any algebraically coherent semi-abelian category---to higher-order Higgins commutators. This is Theorem~\ref{Theorem n Subobjects Lemma} below. Its proof, whose validity strongly depends on Theorem~\ref{Theorem Inequality}, is a variation on the proof Theorem~7.1 in~\cite{acc}. Another key ingredient of the proof is the fact that for any given objects $X_1$, \dots, $X_n$ of $\CC$, the $n$\textsuperscript{th} cross-effect of the identity functor $X_1\cosmash \cdots\cosmash X_n$ is the $(n-1)$\textsuperscript{st} cross-effect of the binary cosmash product functor $X_1\cosmash-\colon \CC \to \CC$. We give a full proof of this result, which (in its most general form) occurs in the currently only partially published manuscript~\cite{HVdL-arXiv} as Lemma~2.20. It, and its proof, are a direct generalisation of Proposition~2.12 in~\cite{Actions}. A short argument based on the $3\times 3$-lemma is given in the proof of \cite[Corollary~6.17c]{BeBou}.

\begin{proposition}
Let $\CC$ be a pointed finitely complete and finitely cocomplete category. Then there is a natural isomorphism 
\[ 
X_1\cosmash \cdots\cosmash X_n \cong (X_1\cosmash -)(X_2|\cdots|X_n)= \Cr_{n-1}(X_1\cosmash -)(X_2, \dots,X_n) 
\]
for objects $X_1$, \dots, $X_n$ in $\CC$. In particular, we find $X_1\cosmash \cdots\cosmash X_n$ as a kernel of the comparison morphism
\[
\musvv{1_{X_1}\cosmash \widehat \pi_{X_2}}{1_{X_1}\cosmash \widehat \pi_{X_n}}\colon X_1\cosmash \coprod_{k=2}^{n}X_k\to \prod_{i=2}^{n}(X_1\cosmash\coprod^n_{\substack{k=2\\k\neq i}}X_k).
\]
\end{proposition}
\begin{proof}
Our strategy is to construct the diagram in Figure~\ref{Figure Cosmash products as cross-effects},
\begin{figure}
\[
\xymatrix@!0@R=7em@C=12em{(X_1\cosmash -)(X_2|\cdots|X_n) \ar@<-.5ex>@{.>}[r]_-{\iota} \ar@{{ |>}->}[d] & X_1\cosmash \cdots\cosmash X_n\ar@{{ |>}->}[d]^-{\iota_{X_1, \dots,X_n}} \ar@<-.5ex>@{.>}[l]_-{\iota''} \ar@{.>}[ld]^-{\iota'}&\\
X_1\cosmash \coprod_{k=2}^{n}X_k \ar@{{ |>}->}[r]^-{\iota_{X_1,\coprod_{k=2}^{n}X_k}} \ar@{}[rd]|{\text{($\natural$)}} \ar[d]_-{u=\musvv{1_{X_1}\cosmash \widehat \pi_{X_2}}{1_{X_1}\cosmash \widehat \pi_{X_n}}}&\coprod_{k=1}^{n}X_k \ar[d]_-{r_{X_1, \dots,X_n}} \ar@{-{ >>}}[r]^-{r_{X_1,\coprod_{k=2}^{n}X_k}} \ar@{}[rd]|(.3){\text{($\sharp$)}} & X_1\times \coprod_{k=2}^{n}X_k\\
\prod_{i=2}^{n}(X_1\cosmash\coprod^n_{\substack{k=2\\k\neq i}}X_k) \ar[r]_-{v}& \prod_{i=1}^{n}\coprod^n_{\substack{k=1\\k\neq i}}X_k \ar[ru]_-{w} &}
\]
\caption{Cosmash products as cross-effects}\label{Figure Cosmash products as cross-effects}
\end{figure}
whose top vertical arrows are kernels of the bottom vertical arrows, and whose middle row is a short exact sequence. Once we have all solid arrows, $\iota$ and $\iota'$ are induced; it then suffices that $v$ is a monomorphism for~$\iota'$ to factor over the kernel of $u$ as an inverse $\iota''$ of $\iota$.

We let 
\[
v\coloneq \musv{0}{\prod_{i=2}^nv_i}\colon \prod_{i=2}^{n}(X_1\cosmash\coprod^n_{\substack{k=2\\k\neq i}}X_k) \to \coprod^n_{k=2}X_k \times \prod_{i=2}^{n}\coprod^n_{\substack{k=1\\k\neq i}}X_k
\]
where
\[
v_i\coloneq \iota_{X_1,\coprod^n_{\substack{\scriptscriptstyle k=2\\ \scriptscriptstyle k\neq i}}X_k}\colon X_1\cosmash\coprod^n_{\substack{k=2\\k\neq i}}X_k \to X_1+\coprod^n_{\substack{k=2\\k\neq i}}X_k.
\]
Then clearly $v$ is a monomorphism, because all of the $v_i$ are. We see that the square~($\natural$) commutes, because $\widehat \pi_{X_1}\comp\iota_{X_1,\coprod_{k=2}^{n}X_k}=0$, while for all $i\geq 2$ we have
\[
\widehat \pi_{X_i}\comp\iota_{X_1,\coprod_{k=2}^{n}X_k}= (1_{X_1}+\widehat \pi_{X_i})\comp\iota_{X_1,\coprod_{k=2}^{n}X_k}=\iota_{X_1,\coprod^n_{\substack{\scriptscriptstyle k=2\\ \scriptscriptstyle k\neq i}}X_k}\comp (1_{X_1}\cosmash \widehat \pi_{X_i})
\]
by naturality of $\iota$. 

We let 
\[
w\coloneq \musv{\lgroup 1_{X_1} \; 0 \; \cdots \; 0\rgroup\circ p_{\coprod_{k=1}^{n-1}X_k}}{p_{\coprod_{k=2}^{n}X_k}}\colon \prod_{i=1}^{n}\coprod^n_{\substack{k=1\\k\neq i}}X_k \to X_1\times \coprod_{k=2}^{n}X_k
\]
where the $p_Y$ denote product projections, and show that the triangle ($\sharp$) commutes:
\begin{align*}
\lgroup 1_{X_1} \; 0 \; \cdots \; 0\rgroup\comp p_{\coprod_{k=1}^{n-1}X_k}\comp 	r_{X_1, \dots,X_n}=\lgroup 1_{X_1} \; 0 \; \cdots \; 0\rgroup
\end{align*}
and
\begin{align*}
	p_{\coprod_{k=2}^{n}X_k}\comp r_{X_1, \dots,X_n}=\widehat\pi_{X_1}.
\end{align*}
This finishes the proof.
\end{proof}

The next theorem is our paper's main result: it extends \cite[Theorem 7.1]{acc} to arbitrary $n\geq 3$.

\begin{theorem}[The $n$ Subobjects Lemma]\label{Theorem n Subobjects Lemma}
Let $\CC$ be an algebraically coherent semi-abelian category. If $K_1$, \dots, $K_n$ are normal subobjects of an object $X$ in $\CC$, where $n\geq 3$, then 
\begin{multline*}
[K_1, \dots,K_n]=[[K_1,K_2],K_3, \dots,K_n]\vee [K_2,[K_1,K_3],K_4, \dots,K_n]\vee\cdots\\
\cdots\vee [K_2,K_3, \dots,K_{n-1},[K_1,K_n]].	
\end{multline*}
\end{theorem}
\begin{proof}
We consider Figure~\ref{Figure A}, 
\begin{figure}
\resizebox{\textwidth}{!}
{\xymatrix@!0@R=3.5em@C=7em{ &0\ar[d]&&0\ar[d]&&0\ar[d]&\\
0 \ar@{.>}[r] & K_1\cosmash K_2\cosmash\cdots\cosmash K_n\ar@{{ |>}.>}[rr] \ar@{{ |>}->}[dd] && A\ar@{{ |>}->}[dd] \ar@<-.5ex>@{.>}[rr] \ar@{{ |>}->}[dd] && K_2\cosmash\cdots\cosmash K_n \ar@{{ |>}->}[dd] \ar@<-.5ex>@{.>}[ll] \ar@{.>}[r] & 0\\\\
0 \ar[r] & K_1\cosmash \coprod_{k=2}^{n}K_k \ar@{{ |>}->}[rr]^-{\jmath_{K_1,\coprod_{k=2}^{n}K_k}} \ar[dd]_-{u} && K_1\flat\coprod_{k=2}^{n}K_k \ar[dd]_-{\musvv{1_{K_1}\flat \widehat \pi_{K_2}}{1_{K_1}\flat \widehat \pi_{K_n}}} \ar@<-.5ex>[rr]_-{\tau^{K_1}_{\coprod_{k=2}^{n}K_k}} && \coprod_{k=2}^{n}K_k \ar@<-.5ex>[ll]_-{\eta^{K_1}_{\coprod_{k=2}^{n}K_k}}\ar[r] \ar[dd]^-{r_{K_2, \dots,K_n}} &0\\\\
0\ar[r]&\prod_{i=2}^{n}(K_1\cosmash\coprod^n_{\substack{k=2\\k\neq i}}K_k) \ar@{{ |>}->}[rr]_-{\prod_{i=2}^{n} \jmath_{1,i}} && \prod_{i=2}^{n}(K_1\flat \coprod^n_{\substack{k=2\\k\neq i}}K_k) \ar@<-.5ex>[rr]_-{\prod_{i=2}^{n}\tau_i} &&\prod_{i=2}^{n}\coprod^n_{\substack{k=2\\k\neq i}}K_k \ar@<-.5ex>[ll]_-{\prod_{i=2}^{n}\eta_i} \ar[r] &0}}
\caption{Constructing the object $A$}\label{Figure A}
\end{figure}
where $\jmath_{1,i}$ is the kernel of the canonical epimorphism
\[
\eta_i\coloneq\eta^{K_1}_{\coprod^n_{\substack{\scriptscriptstyle k=2\\\scriptscriptstyle k\neq i}}K_k}\colon K_1\flat \coprod^n_{\substack{k=2\\k\neq i}}K_k\to \coprod^n_{\substack{k=2\\k\neq i}}K_k,
\]
which is split by the monomorphism $\tau_i$ as in~\ref{Fibration of points}. Note that the upper row in this diagram is exact, because kernels commute with kernels, and any split epimorphism is the cokernel of its kernel. As in~\ref{Subsection Cosmash}, we may see that both rows in the diagram 
\[
\resizebox{\textwidth}{!}
{\xymatrix@R=3em@C=3em{
0\ar[r]&\bigdiamond_{k=2}^{n}(K_1\flat K_k) \ar@{{ |>}->}[rr]^-{\iota_{K_1\flat K_2, \dots,K_1\flat K_n}} \ar@{->}[d]_-{\alpha}&&\coprod_{k=2}^{n}(K_1\flat K_k) \ar@{-{ >>}}[d]_-{\lgroup1_{K_1}\flat \iota_{K_2}\;\cdots\;1_{K_1}\flat \iota_{K_n} \rgroup } \ar@{-{ >>}}[rr]^-{\lambda_{K_1\flat K_2, \dots,K_1\flat K_k}}&& \L E(K_1\flat K_2, \dots,K_1\flat K_n) \ar[r] \ar@{-{ >>}}[d]^-{} &0\\
0\ar[r]&A\ar@{{ |>}->}[rr]_-{} && K_1\flat \coprod_{k=2}^{n} K_k \ar@{-{ >>}}[rr]_-{\lambda^{K_1\flat-}_{K_2, \dots,K_n}}&& \L(K_1\flat E(K_2, \dots,K_n)) \ar[r] &0}}
\]
are short exact sequences. By algebraic coherence in the guise of Proposition~\ref{Proposition AC}, the canonical comparison morphism 
\[
E(K_1\flat K_2, \dots,K_1\flat K_n)\to K_1\flat E(K_2, \dots,K_n)
\]
is a regular epimorphism between split $(n-1)$-cubic extensions. Hence, by Lemma~\ref{Lemma Split Epimorphism of Extensions}, it represents an $n$-cubic extension. Via Lemma~\ref{Lemma D preserves extensions}, this implies that $\alpha$ is a regular epimorphism.

The bottom row in the diagram 
\[
\resizebox{\textwidth}{!}
{\xymatrix@!0@R=5em@C=7em{0\ar[r]&B \pullback \ar@{{ |>}->}[rr]\ar@{->}[d]_-{\beta} && \bigdiamond_{k=2}^{n}(K_1\flat K_k)
\ar@{-{ >>}}[d]_-{\alpha} \ar@<-.5ex>@{->}[rr]_-{\tau^{K_1}_{K_2}\diamond\cdots\diamond\tau^{K_1}_{K_n}} &&K_2\cosmash\cdots\cosmash K_n \ar@<-.5ex>@{->}[ll]_-{\eta^{K_1}_{K_2}\cosmash\cdots\cosmash \eta^{K_1}_{K_n}} \ar[r] \ar@{=}[d] &0\\
0\ar[r]&K_1\cosmash K_2\cosmash\cdots\cosmash K_n\ar@{{ |>}->}[rr]&& A \ar@<-.5ex>@{->}[rr] &&K_2\cosmash\cdots\cosmash K_n \ar@<-.5ex>@{->}[ll] \ar[r] &0}}
\]
is obtained as in Figure~\ref{Figure A}. The diagram's left hand side square is a pullback, so that the morphism $\beta$ is a regular epimorphism. 

For each $k\in\{2, \dots,n\}$, using protomodularity, we deduce from the split short exact sequence
\[
\xymatrix{0 \ar[r] & K_1\cosmash K_k \ar@{{ |>}->}[r] & K_1\flat K_k \ar@<-.5ex>[r] & K_k \ar@<-.5ex>[l] \ar[r] & 0}
\]
of~\cite[Proposition 2.7]{Actions} that $K_1\flat K_k$ is covered by $K_k +(K_1\cosmash K_k)$. Lemma~2.12 of~\cite{HVdL}, the proof of~\cite[Proposition~2.22]{HVdL} and protomodularity together imply that 
\begin{equation}\label{Decomposition}
	\text{$(U+V)\cosmash X$ is covered by $(U\cosmash X) + (V\cosmash X) + (U\cosmash V\cosmash X)$.}
\end{equation}
Using Corollary 2.14 in~\cite{Actions}, we see that $(K_1\flat K_2)\cosmash\cdots\cosmash(K_1\flat K_n)$ is covered by $(K_2+(K_1\cosmash K_2))\cosmash\cdots\cosmash(K_n+(K_1\cosmash K_n))$; we may now use \eqref{Decomposition} iteratively and summandwise into a sum of $3^{n-1}$ terms in which no $\flat$ appear. We write this sum of $3^{n-1}$ terms as $K_2\cosmash \cdots\cosmash K_n+ S$, where $S$ is the sum of the remaining $3^{n-1}-1$ terms.

In the diagram
\[
\resizebox{\textwidth}{!}
{\xymatrix@!0@R=4em@C=6em{0\ar[r]& (K_2\cosmash \cdots\cosmash K_n)\flat S \pullback \ar@{{ |>}->}[rr]^-{\kappa_{K_2\cosmash \cdots\cosmash K_n,S}} \ar@{->}[d]_-{\gamma} && K_2\cosmash \cdots\cosmash K_n+ S
\ar@{-{ >>}}[d]_-{} \ar@<-.5ex>@{->}[rr]_-{\lgroup 1_{K_2\cosmash\cdots\cosmash K_n}\;0\rgroup} &&K_2\cosmash \cdots\cosmash K_n \ar@<-.5ex>@{->}[ll]_-{\iota_{K_2\cosmash\cdots\cosmash K_n}} \ar[r] \ar@{=}[d] &0\\
0\ar[r]&B \ar@{{ |>}->}[rr]&& (K_1\flat K_2)\cosmash\cdots\cosmash(K_1\flat K_n) \ar@<-.5ex>@{->}[rr] &&K_2\cosmash\cdots\cosmash K_n \ar@<-.5ex>@{->}[ll] \ar[r] &0}}
\]
the left-hand square is a pullback square, so that the morphism $\gamma$ is a regular epimorphism. This shows that the cosmash product $K_1\cosmash \cdots\cosmash K_n$ is covered by $(K_2\cosmash \cdots\cosmash K_n)\flat S$, which itself is covered by $S + (K_2\cosmash \cdots\cosmash K_n) \cosmash S$. Considering $K_1$, \dots, $K_n$ as subobjects of $X$ and taking images in $X$ now yields the join decomposition $[K_1, \dots,K_n]=\overline S \vee[ [K_2, \dots,K_n], \overline S]$ of the Higgins commutator where $\overline S$ denotes the image of $S$ in $X$. 

The sum $S$ is of the form 
\begin{multline*}
(K_1\cosmash K_2)\cosmash K_3\cosmash\cdots\cosmash K_n+ K_2\cosmash (K_1\cosmash K_3)\cosmash K_4\cosmash\cdots\cosmash K_n+\cdots\\
\cdots+ K_2\cosmash\cdots \cosmash K_{n-1}\cosmash (K_1\cosmash K_n)+T
\end{multline*}
for some object $T$ containing the remaining terms. The image of $T$ in $X$, denoted by $\overline T$, is a join of higher commutators, each of length between $n-1$ and $2(n-1)$. The only higher commutator in $\overline T$ with length equal to $n-1$ is of the form
\[ 
[[K_1,K_2], \dots, [K_1,K_n] ].
\]
By Proposition~\ref{Higgins properties} and the fact that $[K_1,K_k]\leq [X,K_k]\leq K_k$, we have 
\[
 [[K_1,K_2], [K_1,K_3],\dots, [K_1,K_n]] \leq [[K_1,K_2],K_3,K_4, \dots,K_n ].
\]
All higher commutators in $\overline T$ with length $n$ are of the form
\[ 
[K_2, \dots,K_n ,[K_1,K_k]] 
\]
for some $k\in\{2, \dots, n\}$. It then follows by Proposition~\ref{Higgins properties} and Theorem~\ref{Theorem Inequality} that
\begin{align*}
[K_2, \dots,K_n ,[K_1,K_k]] &\leq [K_2, \dots,K_{k-1},X,K_{k+1}, \dots,K_n ,[K_1,K_k]]\\
&\leq [K_2, \dots,K_{k-1},[K_1,K_k],K_{k+1}, \dots,K_n].
\end{align*}
The higher commutators in $\overline T$ of length $n<l\leq 2(n-1)$, with $l=n+p$, are of the form
\[ 
[K_2, \dots,K_{k_1-1},K_{k_1},[K_1,K_{k_1}],K_{k_1+1}, \dots,K_{k_p-1},K_{k_p},[K_1,K_{k_p}],K_{k_p+1}, \dots,K_n ],
\]
where $2\leq k_1<\cdots<k_p\leq n$.
Using Theorem \ref{Theorem Inequality}, Proposition~\ref{Higgins properties}, and the property \NH, we see that 
\begin{align*}
&[K_2, \dots,K_{k_1-1},K_{k_1},[K_1,K_{k_1}],K_{k_1+1}, \dots,K_{k_p-1},K_{k_p},[K_1,K_{k_p}],K_{k_p+1}, \dots,K_n ]\\
& \leq [K_2, \dots,K_{k_1-1},[K_1,K_{k_1}],K_{k_1+1}, \dots, K_n, X ] \\
& \leq [K_2, \dots,K_{k_1-1},[K_1,K_{k_1}],K_{k_1+1}, \dots, K_n ].
\end{align*}
Hence 
\[
\overline T \leq \bigvee_{k=2}^{n}[K_2, \dots,K_{k-1},[K_1,K_{k}],K_{k+1}, \dots, K_n ],
\]
so that 
\[ 
\overline S = \bigvee_{k=2}^{n}[K_2, \dots,K_{k-1},[K_1,K_{k}],K_{k+1}, \dots, K_n ].
\]
Therefore
\begin{align*}
[K_1, \dots,K_n] &= \overline S \vee[ [K_2, \dots,K_n], \overline S]\\
&= \bigvee_{k=2}^{n}[K_2, \dots,K_{k-1},[K_1,K_{k}],K_{k+1}, \dots, K_n ]\\
 &\quad\vee \bigl[ [K_2, \dots,K_n], \bigvee_{k=2}^{n}[K_2, \dots,K_{k-1},[K_1,K_{k}],K_{k+1}, \dots, K_n ] \bigr] \\
&\leq\bigvee_{k=2}^{n}[K_2, \dots,K_{k-1},[K_1,K_{k}],K_{k+1}, \dots, K_n ], 
\end{align*}
because $\bigvee_{k=2}^{n}[K_2, \dots,K_{k-1},[K_1,K_{k}],K_{k+1}, \dots, K_n ] $ is a normal subobject of the commutator $[K_2, \dots,K_n]$: to see this, combine Proposition~\ref{Higgins properties} (6) with Proposition~\ref{Normality via commutator}. Indeed for all $k\in \{2, \dots,n\}$, Proposition~\ref{Higgins properties}, Theorem \ref{Theorem Inequality} and Proposition~\ref{Normality via commutator} give
\begin{align*}
&[ [K_2, \dots,K_n], [K_2, \dots,K_{k-1},[K_1,K_{k}],K_{k+1}, \dots,K_n ] ]\\
& \leq [K_2, \dots,K_{k-1},K_k, [K_1,K_{k}],K_{k+1}, \dots, K_n ] \\
& \leq [K_2, \dots,K_{k-1},X,[K_1,K_{k}],K_{k+1}, \dots, K_n ]\\
&\leq [K_2, \dots,K_{k-1},[K_1,K_{k}],K_{k+1}, \dots, K_n ].
\end{align*}
The claim follows from~\cite[Proposition 6.2]{MM-NC}. Now, since by Proposition~\ref{Higgins properties} we have
 \begin{align*} 
 [K_2, \dots,K_{k-1},[K_1,K_k], K_{k+1}, \dots,K_n] &=[[K_1,K_k],K_2, \dots,K_{k-1}, K_{k+1}, \dots,K_n]\\
 &\leq [K_1,K_k,K_2, \dots,K_{k-1}, K_{k+1}, \dots,K_n]\\
 &= [K_1, \dots,K_n] 
\end{align*}
for all $k\in \{2, \dots,n\}$, it follows that 
\[
[K_1, \dots,K_n]=\bigvee_{k=2}^{n}[K_2, \dots,K_{k-1},[K_1,K_{k}],K_{k+1}, \dots K_n],
\]
which proves our claim.
\end{proof}

Via Proposition~\ref{Higgins properties} (1), we may use induction on Theorem~\ref{Theorem n Subobjects Lemma} to obtain:

\begin{corollary}\label{Corollary Nilpotency}
In an algebraically coherent semi-abelian category,
\[
[\underbrace{X, \dots,X}_{\text{$n$ terms}}]=[\cdots[[\underbrace{X,X],X], \dots,X}_{\text{$n$ terms}}]
\]
for any object $X$ and any $n\geq 3$.\noproof
\end{corollary}

This result generalises to normal subobjects $K_1$, \dots, $K_n$ of $X$ as follows.

\begin{theorem}\label{Theorem Decomposition into binary commutators}
Let $\CC$ be an algebraically coherent semi-abelian category. Consider normal subobjects $K_1$, \dots, $K_n$ of an object $X$ in~$\CC$, where $n\geq 3$. Then $[K_1, \dots,K_n]$ decomposes as a join of binary commutators:
\begin{equation}\label{Binary decomposition of higher commutator}
[K_1, \dots,K_n]= \bigvee_{\sigma\in S_n}[\cdots[[K_{\sigma(1)},K_{\sigma(2)}],K_{\sigma(3)}], \dots ,K_{\sigma(n)}].
\end{equation}
\end{theorem}
\begin{proof}
We prove this by induction on $n$. For $n=3$, given normal subobjects $K_1$, $K_2$ and $K_3$ of an object $X$, Theorem~\ref{Theorem n Subobjects Lemma} implies
 \[ 
 [[K_2,K_3],K_1]\leq [K_1,K_2,K_3]=[[K_1,K_2],K_3]\vee [K_2,[K_1,K_3]],
 \]
which is also the content of Theorem~7.1 in~\cite{acc}. This proves our claim for $n=3$. Now take $n\geq 4$ and assume that the claim is valid for all Higgins commutators of length strictly less than $n$. 

Using \NH\ and the induction hypothesis, we may see that each term of the decomposition of $ [K_1, \dots,K_n]$ provided by Theorem~\ref{Theorem n Subobjects Lemma} may be further decomposed into binary Higgins commutators. If we set $ L_k=[K_1,K_k]\normal X$, then by the induction hypothesis we have
\begin{align*}
&[K_2, \dots,K_{k-1}, L_k,K_{k+1}, \dots,K_n]\\
&=[ L_k,K_2, \dots,K_{k-1},K_{k+1}, \dots,K_n]\\
&=\bigvee_{\sigma\in P} [ \cdots[[ L_{\sigma(k)},K_{\sigma(2)} ], K_{\sigma(3)}], \dots ,K_{\sigma(k-1)} ], K_{\sigma(k+1)}], K_{\sigma(k+2)}], \dots,K_{\sigma(n)} ]\\
&=\bigvee_{\sigma\in P} [ \cdots[[[K_1, K_{\sigma(k)}] ,K_{\sigma(2)}],K_{\sigma(3)}], \dots ,K_{\sigma(k-1)} ] , K_{\sigma(k+1)}], \dots,K_{\sigma(n)} ]
\end{align*}
for any $k\in \{2, \dots, n\}$, where $P\leq S_n$ is the group of permutations of $\{2, \dots,n\}$.
 Hence $ [K_1, \dots,K_n]$ may be decomposed into a join of binary commutators as 
\begin{equation*}\label{binary decomposition form}
 \bigvee_{k=2}^n \bigvee_{\sigma\in P} [ \cdots[[[K_1, K_{\sigma(k)}] ,K_{\sigma(2)}],K_{\sigma(3)}], \dots ,K_{\sigma(k-1)}], K_{\sigma(k+1)}], \dots,K_{\sigma(n)}].
\end{equation*}
If we write this subobject of $X$ as $[K_1, \dots,K_n]^{-}$, and we denote the join of binary commutators on the right hand side of~\eqref{Binary decomposition of higher commutator} by $[K_1, \dots,K_n]^{+}$, then Proposition~\ref{Higgins properties} tells us that
\[ 
[K_1, \dots,K_n]^{+}\leq [K_1, \dots,K_n].
\]
We now have 
\[
[K_1, \dots,K_n]=[K_1, \dots,K_n]^{-}\leq [K_1, \dots,K_n]^{+}\leq [K_1, \dots,K_n],
\]
which proves that $[K_1, \dots,K_n]=[K_1, \dots,K_n]^{+}$.
\end{proof}

\section*{Acknowledgements}

Thanks to the referee, whose comments and suggestions have substantially improved the presentation of the text and simplified some of the proofs.

%\bibliography{tim}
%\bibliographystyle{amsplain}
%% .bbl:

\end{document}